\documentclass[reqno, 11pt, letterpaper, oneside]{amsart}

\usepackage{geometry}
\usepackage{bbm}
\usepackage[final]{graphicx}
\usepackage{upref}
\usepackage{enumerate}
\usepackage{latexsym}
\usepackage{amssymb}
\usepackage{varioref}
\usepackage{psfrag}
\usepackage[ansinew]{inputenc}
\usepackage[all]{xy}
\usepackage{url}
\usepackage{mycommands}

\newtheorem{theorem}{Theorem}
\newtheorem{lemma}[theorem]{Lemma}
\newtheorem{claim}[theorem]{Claim}

\newtheorem{question}[theorem]{Question}

\theoremstyle{definition}

\theoremstyle{remark}
\newtheorem{remark}[theorem]{Remark}

\newtheorem*{claim*}{Claim}

\long\def\symbolfootnote[#1]#2{\begingroup
\def\thefootnote{\fnsymbol{footnote}}\footnote[#1]{#2}\endgroup}

\begin{document}

\begin{center}

\LARGE Probabilistic one-player Ramsey games\\ via deterministic two-player games
\vspace{2mm}

\Large Michael Belfrage \quad Torsten M\"utze \quad Reto Sp\"ohel\symbolfootnote[2]{The author was supported by the Swiss National Science Foundation, grant 200020-119918.}
\vspace{2mm}

\large
  Institute of Theoretical Computer Science \\
  ETH Z\"urich, 8092 Z\"urich, Switzerland \\
  {\small\tt mbe@student.ethz.ch}, \{{\small\tt muetzet|rspoehel}\}{\small\tt @inf.ethz.ch}
\vspace{5mm}

\small

\begin{minipage}{0.8\linewidth}
\textsc{Abstract.}
Consider the following probabilistic one-player game: The board is a graph with~$n$ vertices, which initially contains no edges. In each step, a new edge is drawn uniformly at random from all non-edges and is presented to the player, henceforth called Painter. Painter must assign one of~$r$ available colors to each edge immediately, where $r\geq 2$ is a fixed integer. The game is over as soon as a monochromatic copy of some fixed graph $F$ has been created, and Painter's goal is to `survive' for as many steps as possible before this happens.

We present a new technique for deriving upper bounds on the threshold of this game, i.e., on the typical number of steps Painter will survive with an optimal strategy. More specifically, we consider a deterministic two-player variant of the game where the edges are not chosen randomly, but by a second player Builder. However, Builder has to adhere to the restriction that, for some real number $d$, the ratio of edges to vertices in all subgraphs of the evolving board never exceeds $d$. We show that the existence of a winning strategy for Builder in this deterministic game implies an upper bound of $n^{2-1/d}$ for the threshold of the original probabilistic game. Moreover, we show that the best bound that can be derived in this way is indeed the threshold of the game if $F$ is a forest. We illustrate our technique with several examples, and derive new explicit bounds for the case when $F$ is a path. 
\end{minipage}

\end{center}

\vspace{5mm}


\section{Introduction}
\label{sec-introduction}

Consider the following probabilistic one-player game: The board is a graph with~$n$ vertices, which initially contains no edges. In each step, a new edge is drawn uniformly at random from all non-edges and is presented to the player, henceforth called Painter. Painter must assign one of~$r$ available colors to each edge immediately, where $r\geq 2$ is a fixed integer. The game is over as soon as a monochromatic copy of some fixed graph $F$ has been created, and Painter's goal is to `survive' for as many steps as possible before this happens. We refer to this as the \emph{online $F$-avoidance game with $r$ colors}. This game was introduced by Friedgut, Kohayakawa, R\"{o}dl, Ruci\'{n}ski, and Tetali~\cite{MR2037068} for the case $F=K_3$ and $r=2$, and further investigated in~\cite{org-lb,org-ub}.

For any graph $F$ and any number $r$ of colors, this game has a threshold $N_0=N_0(F,r,n)$ in the following sense~\cite[Lemma~7]{org-lb}: For any $N = o(N_0)$, there exists a coloring strategy that \aas (asymptotically almost surely, i.e.\ with probability $1-o(1)$ as $n$ tends to infinity) does not create a monochromatic copy of $F$ in the first $N$ steps of the process. On the other hand, if $N=\omega(N_0)$ then \emph{any} online strategy will \aas create a monochromatic copy of $F$ within the first $N$ steps.

Let us point out two bounds on the threshold of the online game that follow from well-known offline results. Clearly, Painter can only  lose the online $F$-avoidance game once the evolving random graph contains a copy of $F$. A well-known result of Bollob\'{a}s~\cite{MR620729} gives a threshold of $n^{2-1/m(F)}$ for the latter property, where
$m(F):=\max_{H\seq F} e_H/v_H$. (Throughout we denote, for any graph $H$, by $e_H$ or $e(H)$ the number of its edges, and by $v_H$ or $v(H)$ the number of its vertices.)

On the other hand, Painter can only survive in the online game as long as the evolving random graph is not $(F,r)$-Ramsey, i.e., does not have the property that every $r$-edge-coloring contains a monochromatic copy of $F$. R\"{o}dl and Ruci\'{n}ski~\cite{MR1249720, MR1276825} proved a threshold of $n^{2-1/m_2(F)}$ for this property, where 
$m_2(F):=\max_{H\seq F} (e_H-1)/(v_H-2)$.
Thus it is clear from the outset that for any $F$ and $r$ the threshold of the online $F$-avoidance game with $r$ colors satisfies \begin{equation} \label{eq:trivialbounds}
n^{2-1/m(F)}\leq N_0(F,r,n)\leq n^{2-1/m_2(F)}\enspace,
\end{equation}
where in fact the lower bound can be interpreted as the threshold of the `game' with $r=1$ colors.

In~\cite{org-lb}, the following lower bound approach was analyzed completely: Denote the colors by $\{1,\ldots,r\}$, and fix suitable subgraphs $H_i\seq F$, $1\leq i\leq r$. Painter's strategy is to color every edge presented with the highest available color $i$ that does not create a monochromatic copy (in color $i$) of the corresponding graph $H_i$ (if no such color is available she uses color~1). We will refer to this approach with $H_1=\cdots=H_r=F$ as the \emph{greedy strategy}, and to the same general approach with an \emph{optimal} choice of $H_i\seq F$, $1\leq i\leq r$, as the \emph{smart greedy strategy}. We say that a strategy \emph{attains} some lower bound $N'(F,r,n)$ on the threshold $N_0(F,r,n)$ if for any $N=o(N')$, \aas it does not create a monochromatic copy of~$F$ in the first $N$ steps of the game.

\begin{theorem}[\cite{org-lb}] \label{thm:smart-greedy}
Let $F$ be a graph that is not a forest, and let $r\geq 2$. Then the threshold of the online $F$-avoidance game with $r$ colors satisfies \begin{equation*}
  N_0(F,r,n) \geq n^{2 - {1}/{\mTwoOnl(F,r)}}\enspace,
\end{equation*}
where $\mTwoOnl(F,r)$ is defined recursively by
\begin{equation*}
	\mTwoOnl(F,r):=
	\begin{cases}
	  \displaystyle \max_{H\seq F}\frac{e_H}{v_H} & \text{if $r = 1$} \enspace, \\
	  \displaystyle \max_{H\seq F}\frac{e_{H}}{v_{H}-2+1/\mTwoOnl(F,r-1)}  & \text{if $r \geq 2$} \enspace.
	\end{cases}
\end{equation*}
This lower bound is attained by the smart greedy strategy.
\end{theorem}

From a qualitative point of view, the main interest of this lower bound is the fact that for every graph $F$ we have
\begin{equation*}
  \lim_{r\to\infty}\mTwoOnl(F,r) = \mTwo(F)\enspace.
\end{equation*}
Thus the threshold of the online game approaches the threshold of the offline setting as the number $r$ of colors increases, cf.~\eqref{eq:trivialbounds}.

As was also pointed out in~\cite{org-lb}, in general the smart greedy strategy is \emph{not} optimal, i.e., there exist non-forests $F$ for which the threshold is strictly higher than~$n^{2 - {1}/{\mTwoOnl(F,r)}}$. We will encounter such an example below.

For the game with two colors and $F$ satisfying a certain precondition, an upper bound matching the lower bound given by Theorem~\ref{thm:smart-greedy} was proved in~\cite{org-ub}, making crucial use of the already mentioned results by R\"{o}dl and Ruci\'{n}ski about offline colorings of random graphs~\cite{MR1276825}. In particular, the following explicit threshold results for complete graphs $K_\ell$ and cycles $C_\ell$ were obtained.

\begin{theorem} [\cite{org-lb, org-ub}]
\label{thm:clique-avoidance}
For any $\ell \geq 3$, the threshold of the online $K_\ell$-avoidance game with $r=2$ colors is
\begin{equation*}
  N_0(K_\ell,2,n)=n^{\(2-\frac{2}{\ell+1}\)\(1-\binom{\ell}{2}^{-2}\)} \enspace.
\end{equation*}
The threshold is attained by the greedy strategy.
\end{theorem}

\begin{theorem} [\cite{org-lb, org-ub}]
\label{thm:cycle-avoidance}
For any $\ell \geq 3$, the threshold of the online $C_\ell$-avoidance game with $r=2$ colors is
\begin{equation*}
  N_0(C_\ell,2,n)=n^{1+1/\ell} \enspace.
\end{equation*}
The threshold is attained by the greedy strategy.
\end{theorem}

\subsection{A new upper bound approach}
\label{sec:intro-general-upper-bound}

In this paper we present a new approach to proving upper bounds on the threshold of the online $F$-avoidance game. In contrast to the approach pursued in~\cite{org-ub}, the ideas in the present paper cover the game with an arbitrary number of colors and extend to graphs for which the smart greedy strategy is not optimal. On the other hand, there seems to be no easy way of recovering all the results of~\cite{org-ub} by our methods, so (at least for the time being) the two approaches should be considered complementary to each other.

Our key idea is to study the deterministic two-player version of the game, which is played by two players called Builder and Painter on a board with some large number $a$ of vertices. In each step, Builder presents an edge, which Painter has to color immediately with one of $r$ available colors. As before, Painter loses as soon as she creates a monochromatic copy of $F$. So far this is exactly the same game as before, except that we replaced `randomness' by the second player Builder. However, we now impose the restriction that Builder is not allowed to present an edge that would create a (not necessarily monochromatic) subgraph $H$ with $e_H/v_H > d$, for some fixed real number $d$. In other words, Builder must adhere to the restriction that the evolving board~$B$ satisfies $m(B)=\max_{H\seq B}e_H/v_H\leq d$ at all times. We will refer to this as the \emph{deterministic $F$-avoidance game with $r$ colors and density restriction $d$ (on a board with $a$ vertices)}.

We say that Builder has a winning strategy in this game (for a fixed graph $F$, a fixed number of colors $r$, and a fixed density restriction $d$) if he can enforce Painter to create a monochromatic copy of $F$ on a board with $a$ vertices for some large enough integer $a$. Conversely, we say that Painter has a winning strategy if she can avoid creating a copy of $F$ on any finite board. (Note that we can think of such a winning strategy as a countably infinite collection of explicit winning strategies, one for every possible board size $a$.)

Our approach is based on the following theorem, which relates the original (probabilistic one-player) online $F$-avoidance game to the deterministic two-player game we just introduced.

\begin{theorem} \label{thm:general-upper-bound}
Let $F$ be a graph with at least one edge, and let $r\geq 2$. If $d>0$ is such that Builder has a winning strategy in the deterministic $F$-avoidance game with $r$ colors and density restriction~$d$, then the threshold of the online $F$-avoidance game with $r$ colors satisfies
\begin{equation*}
  N_0(F,r,n)\leq n^{2-1/d}\enspace.
\end{equation*}
\end{theorem}

The proof of Theorem~\ref{thm:general-upper-bound} is elementary and self-contained. It proceeds by standard small subgraphs type variance calculations, and combines multi-round exposure with the pigeon-hole principle to perform an exhaustive case distinction over all possible strategies Painter can use in the deterministic two-player game.

Let us illustrate our approach with two examples.

\emph{Example 1.} Consider the case where $F=C_\ell$ is a cycle of length $\ell\geq 3$ and $r=2$ colors are available. We will describe an explicit winning strategy for Builder that respects the density restriction $d:=\mTwoOnl(C_\ell,2)= \ell/(\ell-1)$ corresponding to the lower bound given by Theorem~\ref{thm:smart-greedy}. Applying Theorem~\ref{thm:general-upper-bound}, this yields a new elementary  proof of Theorem~\ref{thm:cycle-avoidance} that does not resort to offline coloring results.

\emph{Example 2.} Consider $r=2$ and $F$ the `bowtie' graph consisting of two triangles that are joined by an edge. This is one of the simplest cases in which the smart greedy strategy is \emph{not} optimal:
 According to Theorem~\ref{thm:smart-greedy}, the smart greedy strategy achieves a lower bound of $n^{29/21}= n^{1.380\dots}$, which by an ad hoc Painter strategy can be improved to $n^{60/43}= n^{1.395\dots}$. The results of~\cite{org-ub} do not yield any nontrivial upper bound for this example (the trivial one being $n^{2-1/m_2(F)}=n^{1.5}$, cf.~\eqref{eq:trivialbounds}), but using Theorem~\ref{thm:general-upper-bound} we can derive an upper bound of $n^{86/61}= n^{1.409\dots}$. The threshold of this example therefore satisfies $n^{1.395\dots}\leq N_0(F,2,n) \leq n^{1.409\dots}$.

Our result raises the question whether the best possible upper bound that can be derived from Theorem~\ref{thm:general-upper-bound} is indeed the threshold of the probabilistic game. Similarly, one may ask whether a lower bound counterpart of Theorem~\ref{thm:general-upper-bound} holds, i.e., whether the existence of a winning strategy for \emph{Painter} in the deterministic game with density restriction~$d$ implies a \emph{lower} bound of $n^{2-1/d}$ on the threshold of the probabilistic game. (An affirmative answer to the first question would imply that this is indeed the case.)
While we cannot answer these questions in general, we settle them in the affirmative for the case where $F$ is an arbitrary forest.

\subsection{A threshold result for forests}
Suppose $d$ is of the form $d=k/(k+1)$ for some integer $k\geq 1$. Then the restriction that Builder must not create a subgraph of density more than $d$ is equivalent to requiring that Builder creates no cycles and no components (=trees) with more than $k$ edges. We call this game the \emph{deterministic $F$-avoidance game with $r$ colors and tree size restriction~$k$}.

An elementary proof shows that for any forest $F$ and any integer $r$, Builder has a winning strategy in this game if $k$ is chosen large enough (see~\cite[Prop.~1]{MR2097326}; the result there is stated for $r=2$ but generalizes straightforwardly to any $r\geq 2$). It follows that there is a unique smallest integer $k$ for which Builder has a winning strategy. The next theorem states that the threshold of the probabilistic game for some fixed forest $F$ and integer $r$ is indeed given by this smallest integer.

\begin{theorem} \label{thm:main-result-forests}
Let $F$ be a forest with at least one edge, and let $r\geq 2$. Then the threshold of the online $F$-avoidance game with $r$ colors is
\begin{equation*}
  N_0(F,r,n)=n^{1-1/k^*(F,r)}\enspace,
\end{equation*}
where $k^*(F,r)$ is the smallest integer $k$ for which Builder has a winning strategy in the deterministic $F$-avoidance game with $r$ colors and tree size restriction $k$. The threshold is attained by any winning strategy for Painter in the deterministic game with tree size restriction $k^*(F,r)-1$.
\end{theorem}

Note that Theorem~\ref{thm:main-result-forests} implies that for  the case of forests, the probabilistic aspect of the problem is fully understood, and in order to find the threshold of the probabilistic game for some given $F$ and $r$ it remains to solve the purely deterministic combinatorial problem of determining $k^*(F,r)$.
We will see that, in principle, this can be achieved by a finite calculation. However, this computation becomes intractable already for quite small examples. 

Clearly, it would be desirable to derive closed form expressions for $k^*(F,r)$ for some special families of trees. Unfortunately, we are only able to do so for the trivial case of stars $S_{\ell}$ with $\ell$ edges: For $k\leq r(\ell-1)$ Painter can win the game playing greedily, and for $k\geq r(\ell-1)+1$ Builder easily wins the game by the pigeon-hole principle. Thus for $\ell\geq 1$ and $r\geq 2$ we have
\begin{equation*}
  k^*(S_{\ell},r)=r(\ell-1)+1\enspace.
\end{equation*}

For the case of paths we were able to derive some partial results, which we present in the next section. 

\subsection{Exact values and bounds for paths} We focus on the $P_\ell$-avoidance game with $r=2$ colors, where $P_\ell$ denotes the path with $\ell$ edges. It was shown in~\cite{org-lb} that the greedy strategy yields a lower bound of 
\begin{equation} \label{eq:greedy-lb-paths}
  k^*(P_\ell,2)\geq  \ell +  \lceil \ell/2 \rceil (\ell-1) =: \kOnl(P_\ell,2)\enspace.
\end{equation}

Table~\ref{tab:bounds} lists the exact values of $\kOnl(P_\ell,2)$ and $k^*(P_\ell,2)$ (as defined in Theorem~\ref{thm:main-result-forests}) for all $\ell\leq 13$. The values $k^*(P_\ell,2)$ were determined with the help of a computer, using various branch-and-bound heuristics (cf.~Section~\ref{sec:exact-values}). As we can see, the threshold of the game coincides with the lower bound given by the greedy strategy for $\ell\in\{1,\dots, 13\}\setminus \{8,12\}$, but not for $\ell=8$ and $\ell=12$. This shows that the greedy strategy is not always optimal for trees, answering a question left open in~\cite{org-lb}.

\begin{table}
\begin{center}
\begin{tabular}{|l||c|c|c|c|c|c|c|c|c|c|c|c|c|}\hline
$\ell$            & 1 & 2 & 3 & 4  &  5 &  6 &  7 &           8 &  9 & 10 & 11 &          12 & 13 \\ \hline
$\kOnl(P_\ell,2)$ & 1 & 3 & 7 & 10 & 17 & 21 & 31 & \textbf{36} & 49 & 55 & 71 & \textbf{78} & 97 \\
$k^*(P_\ell,2)$   & 1 & 3 & 7 & 10 & 17 & 21 & 31 & \textbf{39} & 49 & 55 & 71 & \textbf{79} & 97 \\ \hline
\end{tabular}
\end{center}
\caption{Exact values of $\kOnl(P_\ell,2)$ and $k^*(P_\ell,2)$ for $\ell\leq 13$.} \label{tab:bounds}
\end{table}

The observation that for some values of $\ell$ we have $k^*(P_\ell,2) > \kOnl(P_\ell,2)$ raises the question by how much better strategies can improve on the greedy lower bound asymptotically as $\ell\to\infty$. Here we show that the improvement is at least by a constant and at most by a polynomial factor; note that $\kOnl(P_\ell,2)=(1/2+o(1))\cdot\ell^2$.

\begin{theorem} \label{thm:main-result-paths}
We have
\begin{equation*}
  (8/15 + o(1))\cdot \ell^2 \;\leq\; k^*(P_\ell,2) \;\leq\; \Theta\big(\ell^{2\log_2(1+\sqrt{3})}\big)=\Theta\big(\ell^{2.899\dots}\big)
\end{equation*}
as $\ell\to\infty$.
\end{theorem}

The constant $8/15$ is not best possible, and it is entirely conceivable that in fact $k^*(P_\ell,2)=\Omega(\ell^{2+\eps})$ for some $\eps>0$. (We can show such an improvement by a polynomial factor for the vertex-coloring variant of the problem mentioned in the next paragraph.)  The question of the order of magnitude of $k^*(P_\ell,2)$ remains an intriguing open problem.


\subsection{Two extensions} 
Our results extend to the natural generalization of the $F$-avoidance game where Painter is required to avoid a different graph $F_i$ in each color $i$, $1\leq i\leq r$. Both Theorem~\ref{thm:general-upper-bound} and Theorem~\ref{thm:main-result-forests} generalize straightforwardly to these $(F_1,\ldots,F_r)$-avoidance games.

 In~\cite{vertexcase}, a vertex-coloring variant of the online $F$-avoidance game was introduced, and threshold results similar to Theorem~\ref{thm:clique-avoidance} and Theorem~\ref{thm:cycle-avoidance} were proved for an arbitrary number of colors. Both Theorem~\ref{thm:general-upper-bound} and Theorem~\ref{thm:main-result-forests} can easily be adapted to this vertex setting. In fact, we see some hope of proving significantly stronger results for the vertex case, cf.\ the remarks in Section~\ref{sec:outlook}.

\subsection{Related work} Several variations of the deterministic two-player Ramsey game can be found in the literature. The game where no restrictions are imposed on Builder appears already in Beck's paper~\cite{MR701171}, and was introduced independently by Kurek and Ruci\'{n}ski~\cite{MR2152058}. The minimum number of steps needed for Builder to win the game with two colors is called the \emph{online (size) Ramsey number} of $F$~\cite{MR2152058}. Bounds on online Ramsey numbers for various graph classes were proved in~\cite{MR1249704, conlon:online-ramsey, MR2445473}, and exact values for some small graphs were determined in~\cite{MR2445473, MR2152058, pralat-to-appear, MR2381412}. 

Restricted variants of the deterministic game, where for some given graph class $\cH$ Builder has to obey the rule that the board is a graph from $\cH$ at all times, were studied in~\cite{west-et-al, MR2097326, MR2506387}. Specific families $\cH$ considered in these works include forests, planar graphs, $k$-colorable graphs, and graphs with maximum degree $k$. Note that the deterministic game studied in the present paper follows the same framework, with $\cH$ being the family of all graphs $B$ with $m(B)\leq d$.

Another notion that is related to our work is the \emph{Ramsey density} of a given graph $F$, which is defined as the infimum of $m(G)$ over all graphs $G$ that are $(F,2)$-Ramsey. This notion was introduced by Kurek and Ruci\'{n}ski~\cite{MR2152058}. 

Note that Theorem 4 suggests to define, for any graph $F$ and any integer $r$, the \emph{online Ramsey density} of $F$ and $r$ as the infimum over all $d$ for which Builder has a winning strategy in the deterministic $F$-avoidance game with $r$ colors and density restriction $d$. This can be seen as a natural combination of the two well-established concepts of online Ramsey numbers and Ramsey densities. In view of the many open questions revolving around these notions, it is not so surprising that also the online Ramsey densities studied here do not seem to be easily tractable.

\subsection{Organization of this paper}

We prove Theorem~\ref{thm:general-upper-bound} and derive Theorem~\ref{thm:main-result-forests} as a corollary in Section~\ref{sec:general-upper-bound}. In Section~\ref{sec:examples}, we present our examples for non-forests, reproving Theorem~\ref{thm:cycle-avoidance} in an elementary way and deriving new bounds for the bowtie example. After outlining in Section~\ref{sec:calculation} how for any forest $F$ and any $r\geq 2$ the parameter $k^*(F,r)$ can be determined by finite calculation, we focus on the special case of path-avoidance games in Section~\ref{sec:path-avoidance-games}. We discuss how the values in Table~\ref{tab:bounds} were found, and prove the asymptotic bounds stated in Theorem~\ref{thm:main-result-paths}. We conclude the paper by outlining some open questions in Section~\ref{sec:outlook}.

\section{Proof of Theorem~\ref{thm:general-upper-bound} and Theorem~\ref{thm:main-result-forests}}
\label{sec:general-upper-bound}

In order to prove Theorem~\ref{thm:general-upper-bound}, we identify Builder's strategies in the deterministic two-player game with $r$ colors (on a board with some fixed number $a$ of vertices) with $r$-ary rooted trees $\cT$, where each node of such a tree corresponds to an intermediate stage of the game. Specifically, the tree $\cT$ representing a given Builder strategy is constructed as follows: The root of $\cT$ is the empty graph on $a$ vertices. Its $r$ children are the graphs obtained by inserting the first edge of Builder's strategy and coloring it with one of the $r$ available colors. The $r$ children of each of these nodes are in turn obtained by inserting the second edge of Builder's strategy  and coloring it with one of the $r$ colors. (Note that the second edge of Builder's strategy may depend on Painter's decision how to color the first edge, i.e., in general the second edge will be a different one in different branches of $\cT$.) Continuing like this, we construct $\cT$, representing any situation in which Builder stops playing by a leaf of $\cT$.
Thus a node at depth $k$ in $\cT$ is an $r$-colored graph $B$ on $a$ vertices with exactly $k$ edges representing the board of the deterministic game after Painter's $k$-th move if Builder plays according to $\cT$. 

Note that in this formalization, a given tree $\cT$ represents a generic strategy for Builder (in the deterministic game with $r$ colors on a board with $a$ vertices) that may or may not satisfy a given density restriction $d$, and that can be thought of as a strategy for the~`$F$-avoidance' game for any given graph $F$. Clearly, $\cT$ is a legal strategy in the game with density restriction $d$ if and only if $m(B)\leq d$ for (the underlying uncolored graph of) every node $B$ in $\cT$. Moreover, $\cT$ is a winning strategy for Builder in a specific $F$-avoidance game if and only if every leaf of~$\cT$ contains a monochromatic copy of $F$.

Going back to the probabilistic one-player game, we denote the board of the probabilistic game after $N$ moves by $G_N$. Thus $G_N$ is an $r$-colored graph on $n$ vertices with exactly $N$ edges, and the underlying uncolored graph of $G_N$ is uniformly distributed over all graphs on $n$ vertices with $N$ edges. When we say that $G_N$ contains a copy of some $r$-colored graph $B$ (e.g.\ a node of some Builder strategy $\cT$) we mean that there is a subgraph of $G_N$ that is isomorphic to $B$ as a \emph{colored} graph. We write $f \ll g$ for $f=o(g)$, $f\gg g$ for $f=\omega(g)$, and $f\asymp g$ for $f=\Theta(g)$.

Theorem~\ref{thm:general-upper-bound} and Theorem~\ref{thm:main-result-forests} are immediate consequences of the following lemma.

\begin{lemma} \label{lemma:induction-done} Let $r\geq 2$ and $a\geq 1$ be fixed integers, let $d>0$ be a fixed real number, and let $\cT$ represent an arbitrary legal strategy for Builder in the deterministic game with $r$ colors and density restriction $d$ on a board with $a$ vertices.

If $N \gg  n^{2-1/d}$, then regardless of how Painter plays, \aas $G_N$ contains a copy of a leaf of~$\cT$.
\end{lemma}

\begin{proof}[Proof of Theorem~\ref{thm:general-upper-bound}]
By assumption there exists an integer $a=a(F,r,d)$ such that Builder has a winning strategy $\cT$ for the deterministic $F$-avoidance game with $r$ colors and density restriction $d$ on a board with $a$ vertices. As each leaf of $\cT$ contains a monochromatic copy of $F$, applying Lemma~\ref{lemma:induction-done} to~$\cT$ yields that if $N\gg n^{2-1/d}$, \aas $G_N$ contains a monochromatic copy of $F$ regardless of how Painter plays, which is exactly the statement of Theorem~\ref{thm:general-upper-bound}.
\end{proof}

\begin{proof}[Proof of Theorem~\ref{thm:main-result-forests}]
Applying Theorem~\ref{thm:general-upper-bound} with $d=k^*(F,r)/(k^*(F,r)+1)$ immediately yields that $N_0(F,r,n)\leq n^{1-1/k^*(F,r)}$, i.e., Painter will  be forced to create a monochromatic copy of $F$ \aas if $N\gg n^{1-1/k^*(F,r)}$.
On the other hand, as long as $N\ll n^{1-1/k^*(F,r)}$, by standard first moment calculations (see for example~\cite[Section~3]{MR1782847}) \aas $G_N$ contains no cycle, and no tree of size $k^*(F,r)$. In other words, \aas all components of $G_N$ are trees of size at most $k^*(F,r)-1$. Hence by following a winning strategy for the deterministic game with tree size restriction $k^*(F,r)-1$ on a board with $n$ vertices (such a strategy exists by definition of $k^*$), Painter can ensure she will \aas not create a monochromatic copy of~$F$.
\end{proof}

\begin{remark}
Note that after any~$N\gg n$ steps we have $m(G_N)\geq N/n\gg 1$, i.e., the density of the board of the probabilistic game is larger than any constant $d$.
In other words, the random process Painter faces in the probabilistic one-player game only behaves like Builder in the deterministic two-player game (with some constant density restriction $d$) as long as $N=\mathcal{O}(n)$, which is equal to or less than the trivial lower bound $n^{2-1/m(F)}$ if $F$ is a non-forest (cf.~\eqref{eq:trivialbounds}). Therefore `playing just as in the deterministic game' without any additional assumptions only yields a useful guarantee to Painter if $F$ is a forest.
\end{remark}

In order to prove Lemma~\ref{lemma:induction-done}, we shall show the following more technical statement by induction on~$k$.

\begin{claim} \label{claim:induction}
Let $r\geq 2$ and $a\geq 1$ be fixed integers, let $d>0$ be a fixed real number, and let $\cT$ represent an arbitrary legal strategy for Builder in the deterministic game with $r$ colors and density restriction $d$ on a board with $a$ vertices.

If $n^{2-1/d}\ll N  \ll n^2$, then for any integer $k\geq 0$ the following holds.
Regardless of how Painter plays, \aas $G_N$ satisfies one of the following two properties:
\begin{itemize}
\item $G_N$ contains a copy of a leaf of~$\cT$, or
\item there is a node $B$ at depth $k$ in~$\cT$ such that $G_N$ contains  $\Omega(n^{a} (Nn^{-2})^{k})$ many copies of $B$.
\end{itemize}
\end{claim}

The second property of the claim is meaningful since, due to the  assumption that $\cT$ is a legal strategy for Builder in the game with density restriction $d$, we have
\begin{equation*}
  k/a=e_B/v_B \leq m(B)\leq d\enspace,
\end{equation*}
which yields with $N\gg n^{2-1/d}\geq n^{2-a/k}$ that
\begin{equation*}
n^{a}(N n^{-2})^{k} \gg  1\enspace.
\end{equation*}

\begin{proof}[Proof of Lemma~\ref{lemma:induction-done}]
Since $\cT$ has depth at most $\binom{a}{2}$, Lemma~\ref{lemma:induction-done} follows by setting $k:=\binom{a}{2}+1$ in Claim~\ref{claim:induction}.
\end{proof}

It remains to prove Claim~\ref{claim:induction}.

\begin{proof} [Proof of Claim~\ref{claim:induction}]
We proceed by induction on $k$. Clearly, $G_N$ contains $\Theta(n^a)$ copies of the root of $\cT$, i.e., vertex sets of size $a$. This takes care of the induction base.

For the induction step we employ a two-round approach. That is, we divide the process into two rounds of equal length $N/2$ (w.l.o.g.\ we assume $N$ to be even) and analyze these two rounds separately. Specifically, we apply the induction hypothesis and some standard random graph arguments to the edges of the first round, and then show by a variance calculation that, conditional on a `good' first round, the second round turns out as claimed.

By the induction hypothesis, if the graph $G_{N/2}$ does not  contain a copy of a leaf of $\cT$ (in which case we are done), \aas it contains a family of
\begin{equation} \label{eq:M}
M\asymp n^a (Nn^{-2})^{k-1}
\end{equation}
 copies of some graph $B^{-}$ corresponding to a non-leaf node at depth $k-1$ in $\cT$. We label these copies $B^{-}_i$, $1\leq i\leq M$. For a given copy $B^-_i$, consider a vertex pair corresponding to Builder's next move as specified by $\cT$, and call this vertex pair $e_i$. (If this does not define $e_i$ uniquely, we simply fix one possible choice of $e_i$.)
If $e_i$ is an edge of $G_{N/2}$, then clearly $B^-_i$ and $e_i$ form a copy of one of the children of $B^-$ in $\cT$. If this is the case for $n^a (Nn^{-2})^{k}$ many indices~$i$, then by the pigeon-hole principle the color used for the majority of these indices yields a child of $B^-$ for which the inductive claim holds, and we are done with the proof. For the remainder of the proof we assume that this is not the case. Due to $n^a (Nn^{-2})^{k}\ll M$, we can safely ignore indices for which $e_i$ is in $G_{N/2}$; therefore, we assume w.l.o.g.\ that none of the $e_i$ is in $G_{N/2}$.

For $1\leq i\leq M$, let $Z_i$ be the indicator variable for the event that $e_i$ is among the $N/2$ edges drawn in the second round. Let
\begin{equation*}
  Z:=\sum_{i=1}^{M}Z_i\enspace,
\end{equation*}
and note that by the pigeon-hole principle at least $Z/r$ many copies of one of the children of $B^{-}$ in $\cT$ are created. Thus the existence of $B$ as claimed follows if we show that  \aas
\begin{equation} \label{eq:Z}
Z \asymp   n^{a} (N n^{-2})^{k}= n^{v_B} (N n^{-2})^{e_B} \enspace.
\end{equation}
We will do so by the methods of first and second moment.

Using that $N\ll n^2$ we have
\begin{align} \label{eq:indicator-prob}
\Pr[Z_i=1] = \frac{\binom{\binom{n}{2}-N/2-1}{N/2-1}}{\binom{\binom{n}{2}-N/2}{N/2}} &\asymp Nn^{-2}\enspace,
\end{align}
and, conditioning on the first round satisfying the induction hypothesis,
\begin{equation} \label{eq:EZ}
\E[Z] \asymp M\cdot Nn^{-2} \asBy{eq:M}n^{a} (N n^{-2})^{k}= n^{v_B}(N n^{-2})^{e_B} \enspace.
\end{equation}

In the following we slightly abuse notation and write $B$ for the \emph{uncolored} graph formed by $B^-$ and the next edge of Builder's strategy $\cT$.
Let $\cD$ denote the family of all (uncolored) graphs $D$ that can be constructed by considering the union of two edge-intersecting copies of $B$ and removing one edge from the intersection of these two copies. To calculate the variance of $Z$, observe that for pairs with $e_i\neq e_j$ the variables $Z_i$ and $Z_j$ are negatively correlated. Hence such pairs can be omitted, and we have
\begin{equation} \label{eq:variance1}
\begin{split}
\var[Z] =& \sum_{i,j=1}^{M}(\E[Z_i Z_j]-\E[Z_i]\E[Z_j])
\leq \sum_{(i,j)
:\; e_i = e_j}\Pr[Z_i=1\wedge Z_j=1] \\
\asBy{eq:indicator-prob} & \sum_{(i,j)
:\; e_i = e_j} Nn^{-2} \leq M_\cD \cdot \Theta(1) \cdot Nn^{-2} \enspace,
\end{split}
\end{equation}
where $M_\cD$ denotes the total number of copies of graphs $D\in\cD$ in (the underlying uncolored graph of) $G_{N/2}$. By definition of $\cD$, each such graph satisfies
\begin{equation} \label{eq:evD}
\begin{split}
v_D &= 2v_B-v_J\enspace,\\
e_D &= 2 e_B - e_J-1
\end{split}
\end{equation}
for some subgraph $J\seq B$. Moreover, since we assumed that $\cT$ is a legal strategy for Builder in the game with density restriction $d$, we have
\begin{equation*}
  e_J/v_J\leq m(B)\leq d\enspace,
\end{equation*}
which yields with $N\gg n^{2-1/d}\geq n^{2-v_J/e_J}$ that
\begin{equation} \label{eq:J}
\begin{split}
n^{v_J}(N n^{-2})^{e_J} \gg  1\enspace.
\end{split}
\end{equation}

Thus the expected number of copies of $D$ in (the underlying uncolored graph of) $G_{N/2}$  is
\begin{equation*}
\begin{split}
&\binom{n}{v_D}\cdot \Theta(1)\cdot \frac{{\binom{\binom{n}{2}-e_D}{N/2-e_D}}}{\binom{\binom{n}{2}}{N/2}}
\asymp n^{v_D} (N n^{-2})^{e_D} \\
\eqBy{eq:evD} &\; n^{2v_B-v_J} (N n^{-2})^{2e_B - e_J - 1} \llBy{eq:J} n^{2v_B} (N n^{-2})^{2e_B -1}\enspace.
\end{split}
\end{equation*}
As the number of graphs in $\cD$ is bounded by a constant depending only on $a$, it follows with Markov's inequality that
\begin{equation} \label{eq:MJ}
 M_{\cD}\ll n^{2v_B} (N n^{-2})^{2e_B - 1}
\end{equation}
a.a.s. Thus, conditioning on the first round satisfying the induction hypothesis (cf.~\eqref{eq:M} and~\eqref{eq:EZ}) and \eqref{eq:MJ}, we obtain from \eqref{eq:variance1}  that
\begin{align*}
\var[Z] & \llBy{eq:MJ} \(n^{v_B} (N n^{-2})^{e_B}\)^2
\stackrel{\eqref{eq:EZ}}{\asymp} \E[Z]^2\enspace.
\end{align*}
Chebyshev's inequality now yields that \aas the second round satisfies \eqref{eq:Z}. This implies that there is at least the claimed number of copies of one of the children of $B_-$ in $G_N$, as discussed.
\end{proof}

\section{Examples}
\label{sec:examples}

In this section we present our examples illustrating how Theorem~\ref{thm:general-upper-bound} can be applied to derive explicit upper bounds on the threshold of the original (probabilistic one-player) online $F$-avoidance game.

\subsection{Elementary proof of Theorem~\ref{thm:cycle-avoidance}}\label{sec:reprove-cycle-avoidance}

In the following we rederive the threshold of the cycle-avoidance game with $r=2$ colors in a more elementary way. Specifically, we replace the upper bound proof given in~\cite{org-ub} by an application of Theorem~\ref{thm:general-upper-bound}.

\begin{figure}
\centering
\psfrag{tell}{$T_\ell$}
\psfrag{c}{$C$}
\psfrag{p1}{$1$}
\psfrag{pu}{$u=1$}
\psfrag{p2}{$2$}
\psfrag{p3}{$3$}
\psfrag{p4}{$4$}
\psfrag{p5}{$5$}
\psfrag{p6}{$6$}
\psfrag{g1}{$G^1_\ell$}
\psfrag{g2}{$G^2_\ell$}
\psfrag{case1}{Case 1}
\psfrag{case2}{Case 2}
\psfrag{phase1}{\parbox[t]{8cm}{End of phase 1 with a monochromatic copy of the tree $T_\ell$. The dashed edges will be presented in phase 2.}}
\psfrag{phase2}{\parbox[t]{3.9cm}{End of phase 2 with a monochromatic cycle $C$ of length $\ell(\ell-1)$. The dashed edges will be presented in phase 3 ($G^1_\ell$ and $G^2_\ell$ are the graphs including these edges).}}
\psfrag{phase3}{\parbox[t]{6.2cm}{Highlighted in grey are subgraphs of $G^1_\ell$ and $G^2_\ell$ with maximum density $d=\ell/(\ell-1)$.}}
\includegraphics{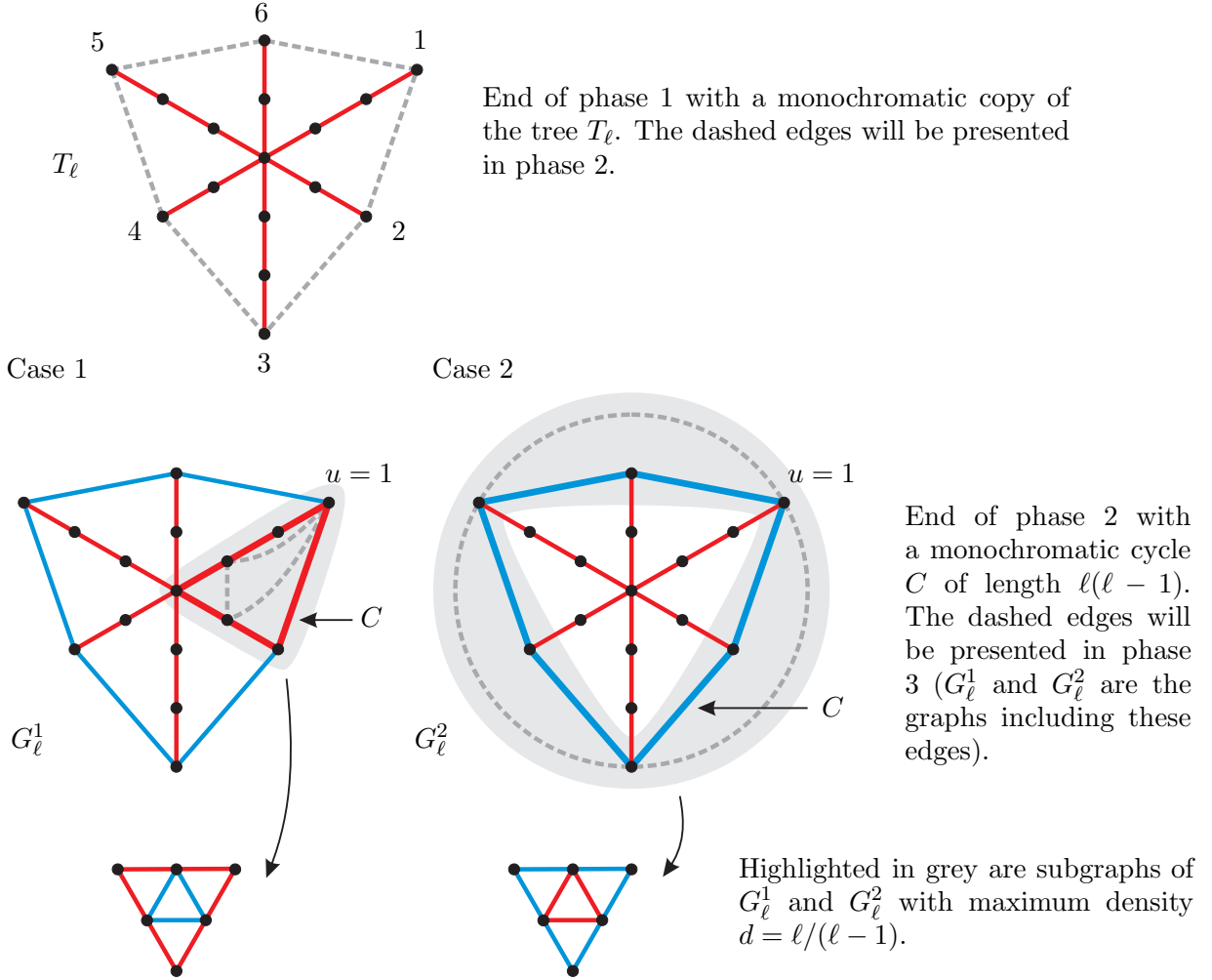}
\caption{Builder strategy to enforce a monochromatic copy of $C_\ell$  in the deterministic game with $r=2$ colors ($\ell=3$).} \label{fig:force-cycle}
\end{figure}

\begin{proof}[Proof of Theorem~\ref{thm:cycle-avoidance}] 
Applying Theorem~\ref{thm:smart-greedy} with $F=C_\ell$ and $r=2$ yields $\mTwoOnl(C_\ell,2)=\ell/(\ell-1)$ and establishes $n^{2-1/\mTwoOnl(C_\ell,2)}=n^{1+1/\ell}$ as a lower bound for the threshold of the probabilistic $C_\ell$-avoidance game with $r=2$ colors. To show that this lower bound given by the greedy strategy is tight, we apply Theorem~\ref{thm:general-upper-bound} and specify a winning strategy for Builder in the deterministic $C_\ell$-avoidance game with $r=2$ colors and density restriction $d:=\mTwoOnl(C_\ell,2)=\ell/(\ell-1)>1$. To define Builder's strategy, let $T_\ell$ denote the tree that is obtained by replacing half of the edges of a star with $\ell(\ell-1)$ edges by paths of length $\ell(\ell-1)/2$, and the other half by paths of length $\ell(\ell-1)/2-1$.

Builder's strategy consists of three phases (cf.~Fig.~\ref{fig:force-cycle}). In the first phase he enforces a monochromatic copy of $T_\ell$ without creating any cycles on the board. As already mentioned in the introduction, it has been proved in \cite[Prop.~1]{MR2097326} that Builder can enforce a monochromatic copy of \emph{any} tree without creating cycles on the board. At the end of the first phase, the monochromatic copy of $T_\ell$ (\wolog we assume that it is a red copy) is contained in some larger tree $T'$. This phase is noncritical as far as the density restriction $d>1$ is concerned, as for any tree $T$ we have $m(T)=e_T/v_T=1-1/v_T<1$.

In the second and third phase Builder joins vertices of the red copy of $T_\ell$ to enforce a monochromatic copy of $C_\ell$. Note that a subgraph $H$ of any non-forest $G$ that maximizes $e_H/v_H$ has the property that each of its edges is contained in a cycle of $G$. As Builder joins only vertices of the red copy of $T_\ell$ in the second and third phase, we can neglect the edges of $T'$ that do not belong to this copy when checking whether the density restriction $d$ is respected. For simplicity we will therefore describe how Builder proceeds with an \emph{isolated} red copy of $T_\ell$ in the second phase (and neglect the embedding of this copy into the larger tree $T'$). Referring to the vertex of maximum degree $\ell(\ell-1)$ of the copy of $T_\ell$ as the root, we label the leaves of this copy with distance $\ell(\ell-1)/2$ from the root by $1,3,5,\ldots,\ell(\ell-1)-1$, and the leaves with distance $\ell(\ell-1)/2-1$ from the root by $2,4,6,\ldots,\ell(\ell-1)$. In the second phase Builder adds $\ell(\ell-1)$ edges, each edge connecting two vertices with successive labels (where the labels $\ell(\ell-1)$ and $1$ are defined to be successive as well). If Painter uses red for one of these edges (Case~1), then a red cycle of length $\ell(\ell-1)$ is created (containing exactly one edge from the second phase). On the other hand, if Painter always uses blue (Case~2), then a blue cycle of length $\ell(\ell-1)$ is created (consisting only of edges from the second phase). In any case, Builder has enforced a monochromatic cycle $C$ of length $\ell(\ell-1)$. Let $u$ denote a vertex with an odd label in this cycle.

In the third phase, Builder connects any two consecutive vertices along $C$ whose distance from $u$ is an integer multiple of $\ell-1$ with a new edge (thus adding $\ell$ edges in total). Clearly, Painter cannot avoid creating a monochromatic copy of $C_\ell$ by the end of this phase.

Depending on the outcome of the second phase (Case~1 or Case~2), the resulting graphs after the third phase, denoted by $G^1_\ell$ and $G^2_\ell$, respectively, are different, and it remains to check that $m(G^1_\ell)\leq d$ and $m(G^2_\ell)\leq d$. Straightforward calculations show that the maximum density of both graphs is indeed bounded by $d$; in fact it is exactly $d$ (cf.\ the bottom part of Fig.~\ref{fig:force-cycle}).
\end{proof}

\subsection{Bounds for the bowtie example} In the following, $F$ denotes the `bowtie' graph consisting of two triangles that are joined by an edge. We prove that the threshold of the online $F$-avoidance game with $r=2$ colors satisfies
$n^{60/43}\leq N_0(F,2,n) \leq n^{86/61}$.

\begin{proof}[Lower bound proof] To prove the claimed lower bound, we consider the following Painter strategy: Color an edge blue if and only if it does not close a blue copy of $F$ and coloring it red would close a red triangle. We will perform a backward analysis, showing that if Painter plays according to this strategy and loses the game with a monochromatic copy of $F$, then the board contains as a subgraph one of a finite family $\cW$ of `witness' graphs, where each graph $W\in\cW$ has a density of $m(W)\geq 43/26$. A standard first moment calculation (see for example~\cite[Section~3]{MR1782847}) yields that for any $N\ll n^{2-26/43}=n^{60/43}$, \aas the board $G_N$ contains no graph from $\cW$ (here we use that the family $\cW$ is finite), which implies the claim.

By definition of the strategy, the game ends with a red copy of $F$. When the last edge in each of the two triangles of this copy was colored red, the alternative for Painter must have been to complete a blue copy of $F$. This implies that six blue edges are adjacent to each red triangle. Each of these blue edges in turn was colored blue only because the alternative was to close a red triangle. Fig.~\ref{fig:bowtie-lb} shows two `nice' possible witness graphs $W_1$ and $W_2$ resulting from this analysis. 
Of course, a witness graph resulting from the above argument is not necessarily as nicely symmetric as $W_1$ and $W_2$. Moreover, some of the blue or red edges could in fact coincide (the graph on the right hand side of Fig.~\ref{fig:bowtie-lb} is such an example). Our analysis therefore yields a fairly large (but finite) family $\cW$ of witness graphs. A straightforward but rather tedious case analysis shows that all graphs $W\in\cW$ satisfy $m(W)\geq 43/26=m(W_1)=m(W_2)$, concluding the proof.
\end{proof}

\begin{proof}[Upper bound proof]
To prove an upper bound of $n^{86/61}$, we describe a strategy for Builder to enforce a monochromatic copy of $F$ in the deterministic $F$-avoidance game with density restriction $d=61/36$. Then Theorem~\ref{thm:general-upper-bound} implies an upper bound of $n^{2-1/d}=n^{86/61}$ for the threshold of the probabilistic game. Builder's strategy consists of two phases. In the first phase he enforces six triangles of the same color (\wolog in blue) in the same way as described in the proof of Theorem~\ref{thm:cycle-avoidance} (cf.\ Fig.~\ref{fig:force-cycle}). In the second phase Builder selects one vertex from each blue triangle and joins those vertices to a copy of $F$. Clearly, Painter cannot avoid creating a monochromatic copy of $F$ by the end of the second phase. Builder's strategy is illustrated in Fig.~\ref{fig:bowtie-ub} (the dashed edges are presented in the second phase), where some parts of the board that were necessary for Builder to enforce the blue triangles are hidden. It is readily checked that the graph $G$ shown in the figure is in fact the densest subgraph of the board (even if the hidden edges are taken into account).
\end{proof}

\begin{figure}
\centering
\psfrag{w1}{$W_1$}
\psfrag{w2}{$W_2$}
\psfrag{w2p}{$W_2'$}
\psfrag{mw12}{$m(W_1)=m(W_2)=\frac{4\cdot 9+7}{4\cdot 6+2}=\frac{43}{26}$}
\psfrag{mw2p}{$m(W_2')=\frac{36}{20}>m(W_2)$}
\includegraphics{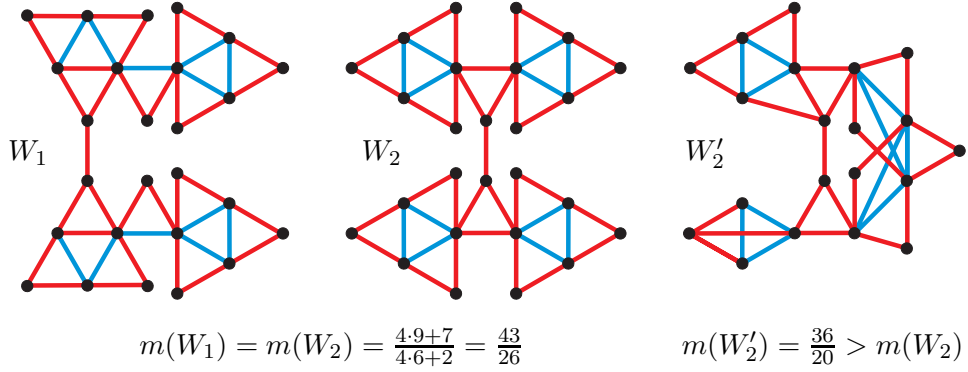}
\caption{Three graphs from the family of witness graphs $\cW$.} \label{fig:bowtie-lb}
\end{figure}

\begin{figure}
\centering
\psfrag{g}{$G$}
\psfrag{mg}{$m(G)=\frac{6\cdot 9+7}{6\cdot 6}=\frac{61}{36}$}
\includegraphics{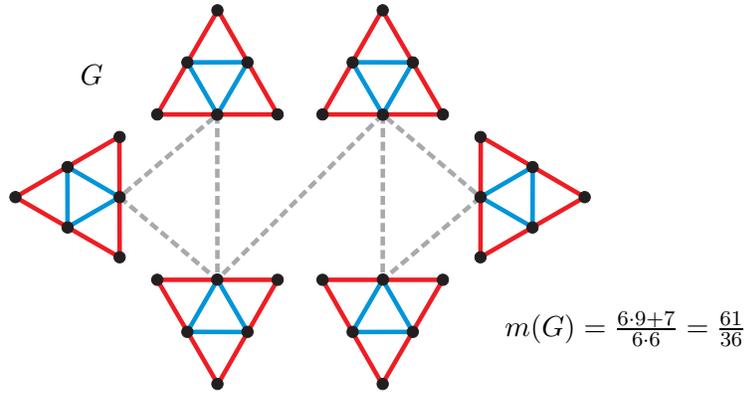}
\caption{Builder strategy to enforce a monochromatic copy of the `bowtie' graph.} \label{fig:bowtie-ub}
\end{figure}

\section{Calculation of $k^*(F,r)$}
\label{sec:calculation}

We outline how the integer $k^*(F,r)$, defined in Theorem~\ref{thm:main-result-forests} as the smallest integer $k$ for which Builder has a winning strategy in the deterministic $F$-avoidance with $r$ colors and tree size restriction $k$, can be found by a finite calculation. First observe that if Builder confronts Painter several times with the decision on how to color a new edge between copies of the same two $r$-edge-colored trees (rooted trees, to be more precise, the root representing the connection vertex), then by the pigeon-hole principle, Painter's decision will be the same in at least a $(1/r)$-fraction of the cases. As a consequence we can assume \wolog that Painter plays \emph{consistently} in the sense that her strategy is determined by a strategy function $\pi$ that maps unordered pairs of $r$-edge-colored rooted trees to the set of available colors $\{1,\ldots,r\}$. For each such strategy function $\pi$ we define the family $\cT^\pi_{r,k}$ as the set of all $r$-edge-colored trees on exactly $k$ edges that Builder can enforce if Painter plays as specified by $\pi$. Observe also that $\cT^\pi_{r,k}$ can be calculated recursively, by joining for each $i=0,\ldots,k-1$ every tree in the family $\cT^\pi_{r,i}$ with every tree in the family $\cT^\pi_{r,k-i-1}$ with a new edge in all possible ways. (The basis for the recursion is the set $\cT^\pi_{r,0}$, which contains only an isolated vertex.) Note that so far this formalism is completely generic, i.e., independent of the specific forest~$F$ Painter tries to avoid.

For a given forest $F$ and a given strategy function $\pi$, it is straightforward to determine the smallest integer $k^*(F,r,\pi)$ such that all components (=trees) of $F$ appear monochromatically in the same color in $\bigcup_{i=0}^{k^*(F,r,\pi)} \cT^\pi_{r,i}$. This value $k^*(F,r,\pi)$ is exactly the lower bound on $k^*(F,r)$ that the strategy function $\pi$ guarantees to Painter. As mentioned in the introduction, an argument given in~\cite{MR2097326} yields an explicit $k_0=k_0(F,r)$ such that for any size restriction $k\geq k_0$ Builder wins the deterministic $F$-avoidance game. Thus we have $k^*(F,r,\pi) \leq k_0$ for all strategy functions $\pi$, and consequently there are only finitely many strategy functions that have to be taken into account (as only the coloring decisions on pairs of rooted trees on at most $k_0$ many edges are relevant). Therefore, $k^*(F,r)$ can  be calculated as the maximum of $k^*(F,r,\pi)$ over finitely many strategy functions $\pi$. 

\section{Path-avoidance games}
\label{sec:path-avoidance-games}

Throughout this section, we study the deterministic game with $F=P_\ell$ (the path with $\ell$ edges) and $r=2$ colors.

\subsection{Exact values for $\ell \leq 13$.}\label{sec:exact-values}
Let us sketch briefly how the exact values of $k^*(P_\ell,2)$ given in Table~\ref{tab:bounds} were determined with the help of a computer. The approach from Section~\ref{sec:calculation} can be easily adapted to compute upper bounds on $k^*(F,r)$ (instead of exact values) by computing only small subsets of the families $\cT^\pi_{r,k}$. The hope is that considering these suffices to enforce a monochromatic copy of $F$ effectively, and that therefore we do not need to branch on too many decisions of Painter (i.e., values of the strategy function $\pi$). If suitable heuristics are used, this approach turns out to be much faster than a full exhaustive enumeration. Moreover, it has the advantage that any resulting upper bound comes with an explicit Builder strategy. Our computer-generated Builder strategies establishing the values given in Table~\ref{tab:bounds} as upper bounds on $k^*(P_\ell,2)$, $\ell\leq 13$ are available on the authors' websites \cite{homepage}, along with a verification routine. 

It remains to prove the matching lower bounds. As discussed in the introduction, for $\ell\in\{1,\ldots,13\}\setminus\{8,12\}$ such lower bounds are provided by the greedy strategy, and it remains to consider the cases $\ell=8$ and $\ell=12$. Here the greedy strategy yields a lower bound of $\kOnl(P_8,2)=36$ and $\kOnl(P_{12},2)=78$, respectively; indeed Fig.~\ref{fig:p8greedy} shows an easy way for Builder to win against a greedy Painter in the $P_8$-avoidance game with tree size restriction $36$ (where Painter greedily colored blue whenever this did not close a blue $P_8$, and all blue edges appeared before all red edges). In the following we propose and analyze a Painter strategy that outperforms the greedy strategy and yields $k^*(P_8,2)\geq 39$ and $k^*(P_{12},2)\geq 79$, thus establishing matching lower bounds for these remaining cases.

\begin{figure}
\centering
\psfrag{eg}{$\kOnl(P_8,2)=8+4\cdot 7=36$}
\includegraphics[scale=0.8]{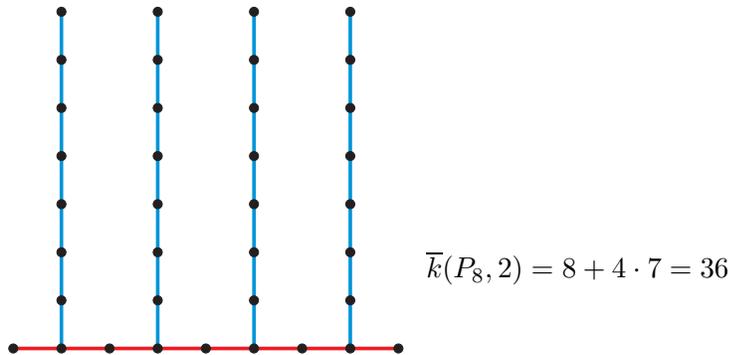}
\caption{Analysis of the greedy strategy for the $P_8$-avoidance game.} \label{fig:p8greedy}
\end{figure}

\begin{proof}[Proof of $k^*(P_8,2)\geq 39$ and $k^*(P_{12},2)\geq 79$] Consider the following generic Painter strategy for the $P_\ell$-avoidance game: Color an edge blue if and only if it is adjacent to a red edge and does not close a blue $P_\ell$. Let us focus on the case $\ell=8$ first. We will perform a backward analysis, showing that if Painter plays according to this strategy and loses the deterministic $P_8$-avoidance game with a monochromatic $P_8$, then Builder must have closed a component (=tree) with at least 39 edges, thus establishing the claimed lower bound for the smallest tree size restriction that guarantees a win for Builder.

By definition of the strategy, the game ends with a red $P_8$. Each edge of this~$P_8$ was colored red either because it was not adjacent to any red edge, or because coloring it blue would have closed a blue $P_8$. In the latter case we call a red edge \emph{heavy}. Note that among any two adjacent edges on the red $P_8$, the one that appeared last must be heavy. We now mark all heavy edges on the red $P_8$. Observe that by unmarking some of these edges again we can always guarantee a pattern of four marked edges where either all of them are disjoint or exactly two of them are adjacent (see Fig.~\ref{fig:p8}).

In the remainder of the proof, we will argue that each heavy edge implies 10 additional edges, and that two adjacent heavy edges imply 11 additional edges. This then proves our claim as we have counted at least $8+\min(4\cdot 10,2\cdot 10+11)=39$ edges in each component (=tree) that might have forced Painter to close a monochromatic $P_8$.

By definition, each heavy edge is adjacent to two blue paths of length $k$ and $7-k$ for some $0\leq k\leq 7$. For any such $k$ the two blue paths contain at least three disjoint blue edges (indicated by dotted ellipses in Fig.~\ref{fig:p8}) that are not adjacent to the central red $P_8$. Again by definition of the strategy, at least one extra red edge must be adjacent to each of these blue edges. We have thus counted at least $7+3=10$ additional edges for each heavy edge.

Assume that two adjacent heavy edges are given, say $(u,v)$ and $(v,w)$. As argued in the previous paragraph, the edge $(u,v)$ guarantees at least 10 additional edges. If the length of the blue path adjacent to $v$ is strictly smaller than 7, then at least one additional blue edge must be adjacent to $w$ (cf.\ the right hand side of Fig.~\ref{fig:p8}). Otherwise, a blue $P_7$ is adjacent to $v$. But this blue path must have been completed \emph{before} both heavy edges appeared, hence not only three but at least four additional red edges are adjacent to this blue $P_7$. In both cases we are guaranteed at least $11$ additional edges for each pair of adjacent heavy edges. This concludes the proof for the case $\ell=8$.

A very similar analysis of the same strategy for $\ell=12$ yields a lower bound of $12 + 16 + 3\cdot17 = 79$ for $k^*(P_{12},2)$, where in the extremal example we have 12 edges coming from the central red $P_{12}$, 16 additional edges implied by an isolated heavy edge, and $3\cdot 17$ additional edges implied by three pairs of adjacent heavy edges.
\end{proof}

\begin{figure}
\centering
\psfrag{pa}{+10 edges}
\psfrag{pb}{+11 edges}
\psfrag{u}{$u$}
\psfrag{v}{$v$}
\psfrag{w}{$w$}
\includegraphics{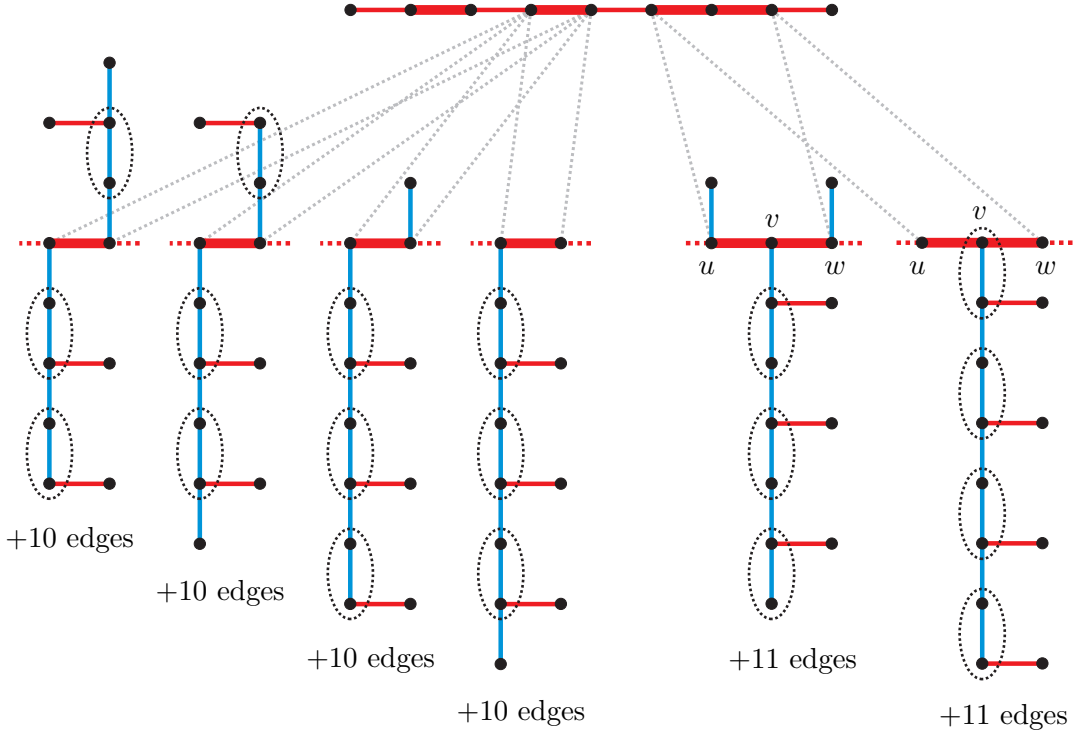}
\caption{Analysis of an optimal strategy for the $P_8$-avoidance game (heavy edges are indicated by bold lines).} \label{fig:p8}
\end{figure}

\subsection{Asymptotic bounds} In this section, we prove the bounds on $k^*(P_\ell,2)$ claimed in Theorem~\ref{thm:main-result-paths} by giving explicit strategies for Painter and Builder.

\begin{proof}[Proof of Theorem~\ref{thm:main-result-paths} (lower bound)]
Consider the following Painter strategy: Color an edge red if and only if (at least) one of the following two conditions holds: (a) the edge is adjacent to a blue edge and it does not close a red $P_4$, (b) coloring the edge blue would close a blue $P_\ell$.

As before, we will perform a backward analysis, showing that if Painter plays according to this strategy and loses the deterministic $P_\ell$-avoidance game with a monochromatic $P_\ell$, then Builder must have closed a component (=tree) with at least $\frac{8}{15}\ell^2+\cO(\ell)$ many edges, thus establishing the claimed lower bound for the smallest tree size restriction that guarantees a win for Builder.

By definition of the strategy, Painter loses with a red $P_\ell$. This $P_\ell$ can be partitioned into at least $\lfloor\ell/5\rfloor$ disjoint red paths of length 4. For each of those $P_4$'s, when the last edge was inserted (and colored red), the alternative for Painter must have been to complete a blue $P_\ell$. Hence $\ell-1$ blue edges are adjacent to the last edge of each red $P_4$. These form two disjoint blue paths of length $k$ and $\ell-k-1$ for some $k\geq 0$.

In the remainder of the proof, we will argue that a blue path of length $k$ guarantees at least $\frac{5}{3}k+\cO(1)$ additional edges that are adjacent to it. This then proves our claim as we have counted at least
\begin{equation*}
  \underbrace{\ell}_{\text{central red $P_\ell$}}
  + \underbrace{\left\lfloor\frac{\ell}{5}\right\rfloor \cdot (\ell-1)}_{\text{blue `almost-copies' of $P_\ell$}}
  + \underbrace{\left\lfloor\frac{\ell}{5}\right\rfloor \cdot \left(\frac{5}{3}\ell+\cO(1)\right)}_{\text{additional edges}}
  = \frac{8}{15} \ell^2+\cO(\ell) \enspace.
\end{equation*}
edges in each component (=tree) that might have forced Painter to close a monochromatic $P_\ell$.

So consider a blue path of length $k$, where $k<\ell$. Each of its edges was colored blue because condition~(a) in the strategy definition was violated, either because the edge was not adjacent to any blue edge, or because coloring it red would have closed a red $P_4$. In the latter case we call a blue edge \emph{heavy}. Note that among any two adjacent edges on the blue $P_k$, the one that appeared last must be heavy. We now mark all heavy edges on the blue $P_k$. If three or more marked edges appear consecutively, then we unmark one or more of the interior edges again such that the resulting pattern only consists of marked paths of length one or two with exactly one unmarked edge in between (see Fig.~\ref{fig:heavy-edges}). Similarly to the proof for the case $\ell=8$ in the previous section, unrolling the history of the heavy edges according to the given Painter strategy yields that each heavy edge implies at least 4 additional edges (3 red and 1 blue), and two adjacent heavy edges imply at least 5 additional edges (4 red and 1 blue, or 3 red and 2 blue). We can thus partition the blue $P_k$ (minus a constant number of border edges) according to the marked edge pattern into segments of length 2 and 3 such that there are at least 4 additional edges for each blue segment of length 2, and at least 5 additional edges for each blue segment of length 3. Hence, in total the Painter strategy guarantees at least $\frac{5}{3}k+\cO(1)$ additional edges for any blue path of length $k$, as required.
\end{proof}

\begin{figure}
\psfrag{p4}{+4 edges}
\psfrag{p5}{+5 edges}
\includegraphics{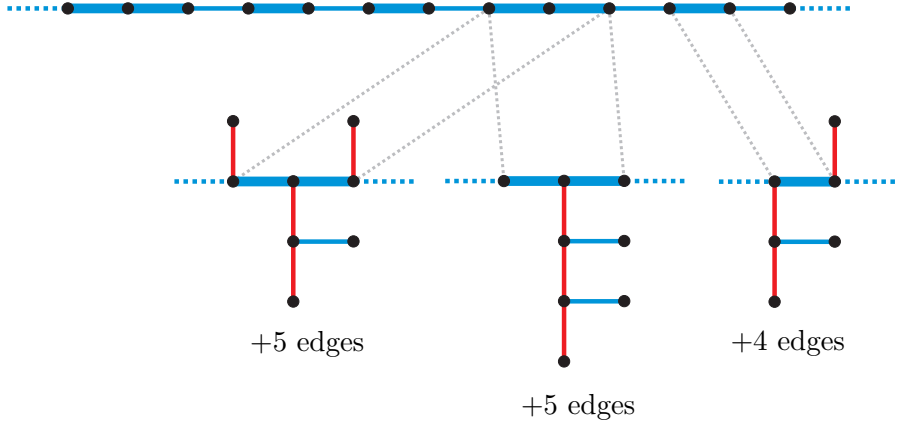}
\caption{Additional edges implied by heavy edges (heavy edges are indicated by bold lines).} \label{fig:heavy-edges}
\end{figure}

\begin{remark} \label{rem:asymmetric}
The above proof relates the $P_\ell$-avoidance game to the asymmetric variant of the game, where Painter's objective is to avoid paths of different lengths in the two colors. For integers~$\ell$ and~$c$, we define $k^*(P_\ell,P_c)$ as the smallest integer~$k$ for which Builder has a winning strategy in the asymmetric $(P_\ell,P_c)$-avoidance game with two colors and tree size restriction $k$.
In the preceding proof we gave a simple strategy guaranteeing $k^*(P_\ell,P_4)\geq (8/3+o(1))\cdot\ell$, and used this to derive a lower bound of $k^*(P_\ell,2)\geq ({8}/{15}+o(1))\cdot \ell^2$ for the symmetric game. 
More generally, if for some fixed $c$ we have a `simple' strategy guaranteeing a lower bound of the form $k^*(P_\ell,P_c)\geq (t+o(1))\cdot\ell$, by the same arguments we can infer a lower bound of $k^*(P_\ell,2)\geq(t/(c+1)+o(1))\cdot \ell^2$ for the symmetric game. This observation might be useful for deriving better lower bounds on $k^*(P_\ell,2)$.
\end{remark}

\begin{proof}[Proof of Theorem~\ref{thm:main-result-paths} (upper bound)]
We describe a strategy for Builder to enforce a monochromatic $P_\ell$ while constructing only components with at most
\begin{equation*}
  {\textstyle\frac{9+5\sqrt{3}}{6}}\cdot\ell^{2\log_2(1+\sqrt{3})}\approx 2.943\cdot\ell^{2.899\ldots}
\end{equation*}
many edges, thus establishing the claimed upper bound for $k^*(P_\ell,2)$. W.l.o.g.\ we may and will assume that Painter plays consistently in the sense of Section~\ref{sec:calculation}.
We say that a rooted edge-colored tree has the \emph{$(r,b)$-property} if its root is adjacent to both a red $P_r$ and a blue $P_b$. Observe that a rooted colored tree with the $(r,b)$-property also has the $(r',b')$-property for all $r'\leq r $ and $b'\leq b$.

To keep track of which colored trees were already created during the game, Builder maintains a table $H$ with row and column indices $\{2^i-1\mid i=0,1,2,\ldots\}=\{0,1,3,7,15,\ldots\}$, where the entry $H_{r,b}$ is either empty or a rooted colored tree with the $(r,b)$-property. A nonempty table entry at $(r,b)$ means that Builder has already enforced a copy of the tree $H_{r,b}$ on the board (and, by the assumption that Painter plays consistently, can enforce as many additional copies of $H_{r,b}$ as he wants).
Initially the table has only a single entry $H_{0,0}=K_1$ (as at the beginning of the game the board consists only of isolated vertices).

In successively filling the table $H$ Builder will maintain the following two properties throughout: There is a \emph{maximal entry} $H_{r,b}$ in the sense that for all other nonempty table entries $H_{r',b'}$ we have $r'\leq r$ and $b'\leq b$. Moreover, all entries $H_{r',0}$ with $r'\leq r$ and all entries $H_{0,b'}$ with $b'\leq b$ are nonempty. We shall refer to the quantity $(r+1)\cdot (b+1)$ as the \emph{covered area} of the table. Clearly, as soon as the covered area exceeds $\ell^2$, Builder has enforced a monochromatic $P_\ell$.

We now describe two Builder operations, each of which exactly doubles the covered area. Each of the two operations consists of two steps and the two operations differ only in their second step, depending on Painter's coloring decision in the first step. The first step is as follows: Builder connects two copies of the maximal entry $H_{r,b}$ at their root vertices. If Painter uses color red, then the resulting tree is used to fill the table entry $H_{2r+1,0}$ (taking one of the end vertices of the red $P_{2r+1}$ as the new root); if Painter uses color blue, then the resulting tree is used to fill the table entry $H_{0,2b+1}$ (taking one of the end vertices of the blue $P_{2b+1}$ as the new root). In the first case (Operation~1) Builder connects a copy of $H_{2r+1,0}$ and a copy of $H_{0,b}$ at their root vertices in the second step; regardless of Painter's coloring decision, the resulting graph can be used as the new maximal entry $H_{2r+1,b}$ (taking one of the end vertices of the added edge as the new root). Otherwise (Operation~2) Builder connects a copy of $H_{0,2b+1}$ and a copy of $H_{r,0}$ at their root vertices in the second step; regardless of Painter's coloring decision the resulting graph can be used as the new maximal entry $H_{r,2b+1}$ (again taking one of the end vertices of the added edge as the new root).
Introducing the abbreviation $e_{r,b}:=e(H_{r,b})$, the sizes of the resulting trees are
\begin{equation} \label{eq:op1}
\begin{split}
  e_{2r+1,0} &= 2 e_{r,b}+1 \\
  e_{2r+1,b} &= e_{2r+1,0}+e_{0,b}+1
\end{split}
\end{equation}
or
\begin{equation} \label{eq:op2}
\begin{split}
  e_{0,2b+1} &= 2 e_{r,b}+1 \\
  e_{r,2b+1} &= e_{0,2b+1}+e_{r,0}+1\enspace,
\end{split}
\end{equation}
respectively.

In this way, depending on Painter's coloring decisions, Builder performs a sequence $\sigma\in\{1,2\}^{\{0,1,\ldots\}}$ of connect-operations (Operation~1 or Operation~2). We obtain two corresponding sequences $\mu,\kappa\in\mathbb{N}^{\{0,1,\ldots\}}$ of sizes $e_{2r+1,0}$ or $e_{0,2b+1}$ of `border'-entries of the table $H$ and of sizes $e_{2r+1,b}$ or $e_{r,2b+1}$ of the maximal table entry in each step, respectively. If e.g.\ $\sigma=(1,2,2,1,1,2,\ldots)$, then we have
\begin{align*}
  \mu    &= (e_{0,0},e_{1,0},e_{0,1},e_{0,3},e_{3,0},e_{7,0},e_{15,0},\ldots)=(0,1,5,15,35,103,239,\ldots) \enspace, \\
  \kappa &= (e_{0,0},e_{1,0},e_{1,1},e_{1,3},e_{3,3},e_{7,3},e_{15,3},\ldots)=(0,2,7,17,51,119,343,\ldots) \enspace,
\end{align*}
cf.~Table~\ref{tab:builder-table}. Note that during the first operation the entry at $(1,0)$ is updated twice (as each operation consists of \emph{two} steps by definition). Even though Builder  wastes some steps in this way, saving them would only complicate the analysis and not change the resulting asymptotic bounds.

\begin{table}
\begin{center}
\begin{tabular}{|r||c|}\hline
$e(H_{r,b})$ & $\, 0 \quad\quad\;\;\, 1 \quad\quad\;\;\, 3 \quad\quad\;\;\; 7$ \\ \hline
\parbox{3mm}{0 \vspace{4.3mm} \\  1 \vspace{4.3mm} \\ 3 \vspace{4.3mm} \\ 7} & 
\parbox{4.5cm}{
$
\begin{xy}
\xymatrix@R=4mm@C=6mm{
  0 \ar[d]_1     &  5 \ar[d]_{2\,}  &  15 \ar[d]_{2\,}    &  239 \ar[ddd]_{2\,\,\,} \\
  1,2 \ar[ru]    &  7 \ar[ru]   &  17 \ar[lld]    &  \\
  35 \ar[rr]^1   &              &  51 \ar[lld]    &  \\
  103 \ar[rr]^1  &              &  119 \ar[uuur]  &  343
}
\end{xy}
$
}
 \\ \hline
\end{tabular}
\end{center}
\caption{Sizes of the resulting trees for the sequence $\sigma=(1,2,2,1,1,2,\ldots)$ of Builder operations to enforce a monochromatic path.} \label{tab:builder-table}
\end{table}

Combining \eqref{eq:op1} and \eqref{eq:op2} yields the general recursion
\begin{align*}
  \mu_0 &= 0 \enspace, \\
  \kappa_0 &= 0 \enspace, \\
  \mu_{i+1} &= 2 \kappa_i+1 \enspace, \\
  \kappa_{i+1} &= \mu_{i+1}+\mu_j+1 \text{  for some $0\leq j\leq i$} \enspace.
\end{align*}

Note that both $\mu$ and $\kappa$ are monotonously increasing, implying that $\kappa_{i+1}\leq\mu_{i+1}+\mu_i+1=2(\kappa_i+\kappa_{i-1})+3$. Solving this recursion yields
\begin{equation} \label{eq:bound-epii}
  \kappa_i \leq {\textstyle\big(\frac{1}{2}+\frac{1}{\sqrt{3}}\big)}(1+\sqrt{3})^i+\underbrace{{\textstyle\big(\frac{1}{2}-\frac{1}{\sqrt{3}}\big)}(1-\sqrt{3})^i - 1}_{\leq 0}
           \leq {\textstyle\big(\frac{1}{2}+\frac{1}{\sqrt{3}}\big)}(1+\sqrt{3})^i
\end{equation}
for all $i\geq 0$.

As both Operation~1 and Operation~2 exactly double the covered area of the table $H$, we have $k^*(P_\ell,2)\leq \kappa_\lambda$, where $\lambda$ is the smallest integer such that $2^\lambda>\ell^2$. Note that the definition of $\lambda$ implies
\begin{equation} \label{eq:lambda}
  \lambda\leq 2\log_2(\ell)+1\enspace.
\end{equation}
Combining these results gives
\begin{equation*}
  k^*(P_\ell,2) \leq \kappa_\lambda
  \leBy{eq:bound-epii} {\textstyle\big(\frac{1}{2}+\frac{1}{\sqrt{3}}\big)}(1+\sqrt{3})^\lambda
  \leBy{eq:lambda} {\textstyle\frac{9+5\sqrt{3}}{6}}\cdot\ell^{2\log_2(1+\sqrt{3})}\enspace.
\end{equation*}
\end{proof}

\section{Open questions}
\label{sec:outlook}

Let us conclude this paper by stating some open questions. We see two main directions for  possible future work. On the one hand, it would be interesting to further investigate the relation between the probabilistic one-player and the deterministic two-player game, with the goal of deriving `abstract' results in the vein of Theorem~\ref{thm:general-upper-bound} and Theorem~\ref{thm:main-result-forests}. In our view, the main open question in this direction is the following. For any graph $F$ and any integer $r$, define the \emph{online Ramsey density} $m_2^*(F,r)$ as 
\begin{multline} \label{eq:m2star}
   m_2^*(F,r):=\inf\Big\{d\in\RR \;\big|\; \text{Builder has a winning strategy in the deterministic } \\ \text{$F$-avoidance game with $r$ colors and density restriction $d$}\Big\}\enspace,
   \end{multline}
i.e., as the infimum over all $d$ for which Theorem~\ref{thm:general-upper-bound} is applicable. 

\begin{question} \label{question}
Is it true that for any graph $F$ and any integer $r\geq 2$, the threshold of the online $F$-avoidance game with $r$ colors is
\begin{equation*}
  N_0(F,r,n) = n^{2-1/{m_2^*(F,r)}} \enspace?
\end{equation*}
\end{question}

For the \emph{vertex-coloring} variant of the problem~\cite{vertexcase}, we have some preliminary results suggesting that the answer to the analogous question is `yes'. We are currently working on a full proof of this conjecture. (For the edge case studied here, we dare not utter such a conjecture and prefer to pose the problem in the open form of Question~\ref{question}.)

The results in this paper answer Question~\ref{question} in the affirmative for the case of cycles and two colors, and for the case of forests and an arbitrary number of colors. In both cases the infimum in~\eqref{eq:m2star} is attained as a minimum, which in particular implies that $m_2^*(F,r)$ is rational.  
In general, we do not even know whether for all $F$ and $r$ the threshold is of the form $N_0(F,r,n)=n^{2-1/x}$ for some rational number $x=x(F,r)$. (Note that a threshold that is not of this form would necessarily have to be sharp, in contrast to what is the case in all known examples; cf.\ the remarks after Theorem~2.1 of~\cite{MR2116574}.)
Let us also point out that Question~\ref{question} remains open for cliques and two colors, even though an explicit threshold function is known for this case (cf.~Theorem~\ref{thm:clique-avoidance}).

A second goal for further research would be to derive further explicit threshold formulas for some specific graphs, either by applying Theorem~\ref{thm:general-upper-bound} or by coming up with new proof techniques. In~\cite{org-ub}, it was conjectured that for cliques and cycles and any number of colors, the lower bound given by Theorem~\ref{thm:smart-greedy} is in fact the threshold of the probabilistic game. Can a matching upper bound be derived from Theorem~\ref{thm:general-upper-bound}?   Currently we are unable to do so even for the simplest open case $F=K_3$, $r=3$. Finally, it would be interesting to improve upon the bounds in Theorem~\ref{thm:main-result-paths} by exhibiting better strategies for Painter or Builder, possibly arguing via the asymmetric case as outlined in Remark~\ref{rem:asymmetric}.

\bibliographystyle{plain}
\bibliography{refs}

\end{document}